\documentclass[11pt]{article}
\usepackage{amssymb,amsmath,amsthm}
\usepackage{enumerate}
\usepackage{pgf,tikz}
\usepackage{mathrsfs}
\usepackage{rotating}
\usepackage{pdflscape}
\usetikzlibrary{arrows}

\let\oldmarginpar\marginpar
\renewcommand\marginpar[1]{\-\oldmarginpar[\raggedleft\footnotesize #1]%
{\raggedright\footnotesize #1}}

\newtheorem{theorem}{Theorem}
\newtheorem{corollary}[theorem]{Corollary}
\newtheorem{obs}[theorem]{Observation}
\newtheorem{lemma}[theorem]{Lemma}
\newtheorem{proposition}[theorem]{Proposition}
\newtheorem{conjecture}[theorem]{Conjecture}
\newtheorem{question}{Question}

\numberwithin{equation}{section}

\newcommand{\pr}{\mathfrak{pr}}

\newcommand{\john}[1]{{\color{red} \sf John: [#1]}}

\newcommand{\vanish}[1]{}

\usepackage{verbatim}

\begin{document}

\title{Minimum Coprime Labelings for Operations on Graphs}

\author{
John Asplund\\
{\small Department of Technology and Mathematics,
Dalton State College,} \\
{\small Dalton, GA 30720, USA} \\
{\small jasplund@daltonstate.edu}\\
\\
N. Bradley Fox\\
{\small Department of Mathematics and Statistics, Austin Peay State University} \\
{\small Clarksville, TN 37044} \\
{\small foxb@apsu.edu}
 }

\date{}
\maketitle

\begin{abstract}
A prime labeling of a graph of order $n$ is a labeling of the vertices with the integers $1$ to~$n$ in which adjacent vertices have relatively prime labels.  A coprime labeling maintains the same criterion on adjacent vertices using any set of distinct positive integers.  In this paper, we consider several families of graphs or products of graphs that have been shown to not have prime labelings and answer the natural question of how to label the vertices while minimizing the largest value in its set of labels.
\end{abstract}

\section{Introduction}
Consider $G$ to be a simple graph with vertex set $V$ in which $|V|=n$ and edge set $E$.  
Throughout this paper, we let $p_i$ be the $i^{\rm th}$ prime number.
A \textit{prime labeling} of $G$ is a labeling of $V$ using the distinct integers $\{1,\ldots,n\}$ such that the labels of any pair of adjacent vertices are relatively prime; if such a labeling exists for a graph, we say that graph is prime.  A graph $G$ is called a \textit{prime graph} if such a labeling exists on $G$.  More generally, a \textit{coprime labeling} of $G$ uses distinct labels from the set $\{1,\ldots, m\}$ for some integer $m\geq n$ such that adjacent labels are relatively prime.  
The minimum value $m$ for which $G$ has a coprime labeling is defined as the \textit{minimum coprime number}, denoted as $\pr(G)$, and a coprime labeling of $G$ with largest label being $\pr(G)$ is called a \textit{minimum coprime labeling} of $G$.  A prime graph therefore has $\pr(G)=n$ as its minimum coprime number.

The concept of a prime labeling of a graph was first developed by Roger Entriger and introduced in~\cite{TDH} by Tout, Dabboucy, and Howalla.  
While most research has revolved around finding prime labelings for various classes of graphs, our focus is on the problem of determining the minimum coprime number for graphs that have been shown to not be prime, a question that was previously studied for complete bipartite graphs $K_{n,n}$ by Berliner et al.~\cite{Berliner}.  A dynamic survey of results on the 35 year history of prime labelings is given by Gallian in~\cite{Gallian}.

It was conjectured by Entriger that all trees have prime labelings, and many classes of trees such as paths, stars, caterpillars, complete binary trees, and spiders have been shown by Fu and Huang in~\cite{FH} to be prime.  Salmasian \cite{S2} showed that for every tree $T$ with $n$ vertices ($n\geq 50$), $\pr(T)\leq 4n$.  Pikhurke \cite{P} improved this by showing that for any integer $c>0$, there is an $N$ such that for any tree $T$ of order $n>N$, $\pr(T)<(1+c)n$. Additionally, many graphs that are not trees have been proven to be prime, including cycles, helms, fans, flowers, and books for all sizes; see~\cite{DLM}, \cite{SDE}, and \cite{SY}. 
There is a large collection of graphs whose primality depends on the size of its vertex set.  
The complete graph $K_n$, for example, is clearly prime only if $n\leq 3$.  Additionally, the wheel graph $W_n$, which consists of a cycle of length $n$ where each vertex on the cycle is adjacent to a central vertex, is prime if and only if $n$ is even.  Section~\ref{wheels} examines the minimum coprime number for complete graphs with at least $4$ vertices and wheel graphs in which $n$ is odd.

While paths and cycles are known to be prime, combinations of these through common graph operations often result in a graph that is not prime.  Recall the \textit{disjoint union} of graphs $G$ and $H$ is the graph $G\cup H$ with vertex set $V(G)\cup V(H)$ and edge set $E(G)\cup E(H)$.  Deretsky et al.~\cite{DLM} proved that $C_{2k}\cup C_n$ is prime for all integers $k$ and $n$.  However, if both of the cycles are of odd length, their disjoint union is not prime.  Determining the minimum coprime number in this case of the union of odd cycles will be our first focus in Section~\ref{unions}.  
We also consider the union of the complete graph with either a path or the star graph (denoted as $S_n$ where $n$ is the number of degree $1$ vertices, also called pendant vertices).  Youssef and El Sakhawi~\cite{EY} studied these two union graphs, concluding that $K_m\cup P_n$ is prime if and only if $1\leq m\leq 3$ or $m=4$ with $n\geq 1$ being odd.  They also determined $K_m\cup S_n$ is prime if and only if the number of primes less than or equal to $m+n+1$ is at least $m$.  
A graph operation that can be applied to complete graphs is the \textit{corona} operation, which is defined as follows. 
The corona of a graph $G$ with a graph $H$, in which $|V(G)|=n$, is denoted by $G\odot H$ and is obtained by combining one copy of $G$ with $n$ copies of $H$ by attaching the $i^{\rm th}$ vertex in $G$ to every vertex within the $i^{\rm th}$ copy of $H$. 
In particular, we examine the corona of a complete graph on $n$ vertices with an empty graph on $1$ or $2$ vertices.
We examine the non-prime cases for these unions and coronas to determine their minimum coprime number in
Section~\ref{unions}.

Given a graph $G$, the $k^{\rm th}$ \textit{power} of $G$, denoted $G^k$, is defined as the graph with the same vertex set as $G$ but with an edge between each $u,v\in V(G)$ for which $d(u,v)\leq k$ in $G$.  Here the value $d(u,v)$ is the distance between $u$ and $v$, or the length of the shortest path between the two vertices.  The square of paths and cycles, denoted as $P_n^2$ and $C_n^2$, were shown not to be prime by Seoud and Youssef in~\cite{SY}.  This implies that higher powers of these are also not prime since $G^k$ is a subgraph of $G^\ell$ for integers $k\leq \ell$.  Section~\ref{powers} explores the minimum coprime numbers for the square and cube of both the path and cycle graphs.

The \textit{join} of two disjoint graphs $G$ and $H$, denoted as $G+H$, consists of a vertex set $V(G)\cup V(H)$ with an edge added to connect each vertex in $G$ to those in $H$, resulting in an edge set $E(G)\cup E(H)\cup \{uv: u\in V(G), v\in V(H)\}$.  Seoud, Diab, and Elsahawi~\cite{SDE} studied the primality of the join of paths with the empty graph $\overline{K}_m$ on $m$ isolated vertices.  They proved that $P_n+\overline{K}_2$ is prime if and only $n=2$ or $n$ is odd, and that the join graph $P_n+\overline{K}_m$ is not prime for all $m\geq 3$.  Since these two classes of graphs are subgraphs of $P_n+P_2$ and $P_n+P_m$ respectively, the join of paths follow similar criteria for not being prime.  Analogous reasoning applies for the join of two cycles or of a path and a cycle.   We will find the minimum coprime number for certain cases of these join graphs depending on the relationship between $m$ and $n$ within Section~\ref{joins}.

Our final section concludes with open problems for further research.  While there are still many unanswered questions about the primality of graphs such as conjectures on trees and unicyclic graphs being prime, we include open questions regarding the minimum coprime number of particular classes of graphs.

\section{Complete Graphs and Wheels}\label{wheels}

Consider the complete graph $K_n$ on $n$ vertices.  It is easy to see that $K_n$ is prime if and only if $n\leq 3$.  We first examine the minimum coprime number for the complete graph with $4$ or more vertices.

\begin{proposition}\label{complete}
Let $n\geq 4$.  The minimum coprime number of $K_n$ is $\pr(K_n)=p_{n-1}$.
\end{proposition}

\begin{proof}
Since each vertex is adjacent to every other vertex, only prime numbers and $1$ can be used as vertex labels to keep each pair of vertices relatively prime.  Thus, we can label the graph with a minimum coprime labeling by using the first $n-1$ primes along with $1$. 
\end{proof}

Next we consider the wheel graph $W_n$.  We name the vertices $v_1,\ldots,v_n$ as shown in Figure~\ref{wheel} with $v_1$ representing the center vertex and the remaining $v_i$ listed in clockwise order with $v_2$ being adjacent to $v_{n+1}$. Lee, Wui, and Yeh~\cite{LWY} demonstrated that $W_n$ is prime if and only if $n$ is even.  The following result determines the minimum coprime number for the odd case.

\begin{figure}
\begin{center}
\includegraphics[scale=1]{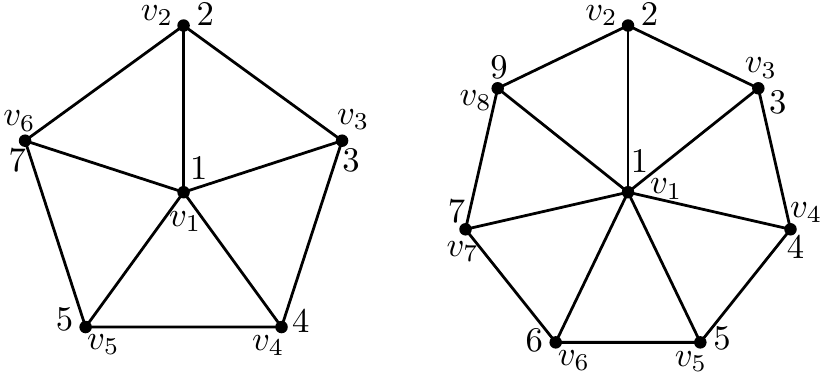}
\caption{The graph $W_n$}\label{wheel}
\end{center}
\end{figure}

\begin{proposition}
Let $n$ be odd.  Then the minimum coprime number of $W_n$ is $\pr(W_n)=n+2$.
\end{proposition}

\begin{proof}
We label the vertices $v_i$ as $i$ for each $i\in\{1,\ldots, n\}$ and $v_{n+1}$ as $n+2$.  This labeling is coprime since each adjacent pair of labels falls into one of four cases: the center label $1$ is in the pair, the labels are consecutive integers, the labels are $n$ and $n+2$ which are consecutive odd numbers, or the labels are $2$ and $n+2$, where again $n$ is assumed to be odd.  The labeling is a minimum coprime labeling since it was proven in \cite{LWY} to be not prime, hence a labeling with maximum label being the the number of vertices, $n+1$, is impossible to achieve.
\end{proof}

\section{Disjoint Union of Graphs and Corona}\label{unions}

The following observation is straightforward from the definitions of prime and coprime labelings and will be useful for an upcoming proof in this section. 

\begin{obs}\label{primeFromPrime}
If $G$ is not prime, then a spanning supergraph of $G$ is not prime. If $G$ is a prime, then any spanning subgraph of $G$ is also prime. 
\end{obs}

The disjoint union of two graphs has been shown to be prime for a variety of graphs under certain conditions. In~\cite{DLM}, the disjoint union of cycles was shown to be prime when at least one of the cycles has an even number of vertices; that is, $C_{2k}\cup C_n$ is prime for all positive $k,n\in \mathbb{Z}$.  
Examining the case in which both cycles are odd, we see that we only need to increase the largest label by one to achieve a minimum coprime labeling.

\begin{theorem}
For all $\ell\geq k\geq 1$, the minimum coprime number of the disjoint union of two odd-length cycles is $$\pr(C_{2k+1}\cup C_{2\ell+1})=2(k+\ell)+3.$$
\end{theorem}

\begin{proof}
Let $p$ be the smallest prime for which $p$ does not divide $2k+2$, which is at least $3$ since $2k+2$ is even.  We break down the labeling into two cases depending on whether $p=3$ or $p>3$.  First assume that $p=3$, so $3\nmid (2k+2)$.  We label the vertices of the graph as shown in Figure~\ref{case1}.  Most of the edges have endpoints with labels of the form $\{m,m+1\}$, which are relatively prime as consecutive integers.  The edges that include the label $1$ as an endpoint or that connect $2$ to an odd label also clearly have relatively prime endpoints.  Since $2(k+\ell)+1$ and $2(k+\ell)+3$ are consecutive odd numbers, they are also relatively prime.  The final edge to consider is between the labels $3$ and $2k+2$, which are relatively prime based on the assumption that $3$ does not divide $2k+2$. Thus, the labeling provided in Figure~\ref{case1} is a coprime labeling of $C_{2k+1}\cup C_{2\ell+1}$ for this case of $p=3$.

\begin{figure}[htb]
\begin{center}
\begin{tikzpicture}[line cap=round,line join=round,>=triangle 45,x=1.0cm,y=1.0cm]
\clip(-2.8806360895220697,1.578329584249893) rectangle (9.461172195389738,8.095095610796193);
\draw [line width=1.2pt] (1.,3.)-- (-1.,3.);
\draw [line width=1.2pt] (-1.,3.)-- (-2.2469796037174667,4.56366296493606);
\draw [line width=1.2pt] (-2.2469796037174667,4.56366296493606)-- (-1.8019377358048378,6.513518789299707);
\draw [line width=1.2pt] (-1.8019377358048378,6.513518789299707)-- (0.,7.381286267534822);
\draw [line width=1.2pt] (0.,7.381286267534822)-- (1.8019377358048383,6.513518789299706);
\draw [line width=1.2pt,dash pattern=on 3pt off 3pt] (1.8019377358048383,6.513518789299706)-- (2.2469796037174667,4.563662964936059);
\draw [line width=1.2pt] (2.2469796037174667,4.563662964936059)-- (1.,3.);
\draw [line width=1.2pt] (6.5,3.)-- (5.,3.);
\draw [line width=1.2pt] (5.,3.)-- (3.850933335321533,3.9641814145298095);
\draw [line width=1.2pt] (3.850933335321533,3.9641814145298095)-- (3.5904610688211385,5.441393044048121);
\draw [line width=1.2pt] (3.5904610688211385,5.441393044048121)-- (4.3404610688211385,6.740431149724778);
\draw [line width=1.2pt] (4.3404610688211385,6.740431149724778)-- (5.75,7.25346136471328);
\draw [line width=1.2pt] (5.75,7.25346136471328)-- (7.159538931178862,6.740431149724777);
\draw [line width=1.2pt] (7.159538931178862,6.740431149724777)-- (7.9095389311788615,5.4413930440481195);
\draw [line width=1.2pt,dash pattern=on 3pt off 3pt] (7.9095389311788615,5.4413930440481195)-- (7.6490666646784655,3.9641814145298078);
\draw [line width=1.2pt] (7.6490666646784655,3.9641814145298078)-- (6.5,3.);
\draw [line width=1.2pt] (6.5,-4.)-- (5.,-4.);
\draw [line width=1.2pt] (5.,-4.)-- (3.850933335321533,-3.0358185854701905);
\draw [line width=1.2pt] (3.850933335321533,-3.0358185854701905)-- (3.590461068821138,-1.5586069559518787);
\draw [line width=1.2pt] (3.590461068821138,-1.5586069559518787)-- (4.3404610688211385,-0.25956885027522114);
\draw [line width=1.2pt] (4.3404610688211385,-0.25956885027522114)-- (5.75,0.25346136471328107);
\draw [line width=1.2pt] (5.75,0.25346136471328107)-- (7.159538931178862,-0.2595688502752225);
\draw [line width=1.2pt] (7.159538931178862,-0.2595688502752225)-- (7.9095389311788615,-1.5586069559518803);
\draw [line width=1.2pt,dash pattern=on 3pt off 3pt] (7.9095389311788615,-1.5586069559518803)-- (7.649066664678466,-3.0358185854701913);
\draw [line width=1.2pt] (7.649066664678466,-3.0358185854701913)-- (6.5,-4.);
\draw [line width=1.2pt] (-0.5,-4.)-- (-2.,-4.);
\draw [line width=1.2pt] (-2.,-4.)-- (-3.149066664678467,-3.0358185854701905);
\draw [line width=1.2pt,dash pattern=on 3pt off 3pt] (-3.149066664678467,-3.0358185854701905)-- (-3.409538931178862,-1.5586069559518787);
\draw [line width=1.2pt] (-3.409538931178862,-1.5586069559518787)-- (-2.6595389311788615,-0.25956885027522114);
\draw [line width=1.2pt] (-2.6595389311788615,-0.25956885027522114)-- (-1.25,0.25346136471328107);
\draw [line width=1.2pt] (-1.25,0.25346136471328107)-- (0.15953893117886264,-0.2595688502752225);
\draw [line width=1.2pt] (0.15953893117886264,-0.2595688502752225)-- (0.909538931178862,-1.5586069559518803);
\draw [line width=1.2pt,dash pattern=on 1pt off 1pt] (0.909538931178862,-1.5586069559518803)-- (0.6490666646784662,-3.0358185854701913);
\draw [line width=1.2pt] (0.6490666646784662,-3.0358185854701913)-- (-0.5,-4.);
\begin{scriptsize}
\draw [fill=black] (-1.,3.) circle (2.5pt);
\draw[color=black] (-1.0605462114141779,2.6432201221073193) node {$1$};
\draw [fill=black] (1.,3.) circle (2.5pt);
\draw[color=black] (1.2697065195633166,2.5522038368203597) node {$2k+1$};
\draw [fill=black] (2.2469796037174667,4.563662964936059) circle (2.5pt);
\draw[color=black] (2.54393451358076,4.645578398420429) node {$2k$};
\draw [fill=black] (1.8019377358048383,6.513518789299706) circle (2.5pt);
\draw[color=black] (1.9614302877442147,6.720749702963105) node {$5$};
\draw [fill=black] (0.,7.381286267534822) circle (2.5pt);
\draw[color=black] (0.17751109611979396,7.576302784660524) node {$4$};
\draw [fill=black] (-1.8019377358048378,6.513518789299707) circle (2.5pt);
\draw[color=black] (-1.788440666078547,6.838952960020497) node {$3$};
\draw [fill=black] (-2.2469796037174667,4.56366296493606) circle (2.5pt);
\draw[color=black] (-2.424328324646777,4.026667658469103) node {$2k+2$};
\draw [fill=black] (5.,3.) circle (2.5pt);
\draw[color=black] (4.928325360467805,2.6250168650499273) node {$2$};
\draw [fill=black] (6.5,3.) circle (2.5pt);
\draw[color=black] (6.712480379723539,2.6340005797629677) node {$2(k+\ell)+3$};
\draw [fill=black] (7.6490666646784655,3.9641814145298078) circle (2.5pt);
\draw[color=black] (8.205147458429686,3.5351797179195223) node {$2(k+\ell)+1$};
\draw [fill=black] (7.9095389311788615,5.4413930440481195) circle (2.5pt);
\draw[color=black] (8.31436700077404,5.664960793634375) node {$2k+8$};
\draw [fill=black] (7.159538931178862,6.740431149724777) circle (2.5pt);
\draw[color=black] (7.367797633789652,6.9755953017665915) node {$2k+7$};
\draw [fill=black] (5.75,7.25346136471328) circle (2.5pt);
\draw[color=black] (5.438252385706096,7.503489756430957) node {$2k+6$};
\draw [fill=black] (4.3404610688211385,6.740431149724778) circle (2.5pt);
\draw[color=black] (4.073008106401692,6.9573920447091995) node {$2k+5$};
\draw [fill=black] (3.5904610688211385,5.441393044048121) circle (2.5pt);
\draw[color=black] (2.9444061688433854,5.501367307749159) node {$2k+4$};
\draw [fill=black] (3.850933335321533,3.9641814145298095) circle (2.5pt);
\draw[color=black] (3.2902680529338335,3.5169764608621303) node {$2k+3$};
\draw [fill=black] (5.,-4.) circle (2.5pt);
\draw[color=black] (-3.5359533435881834,9.83350665977712) node {$2$};
\draw [fill=black] (6.5,-4.) circle (2.5pt);
\draw[color=black] (-2.9898556318664222,9.83350665977712) node {$2(k+\ell)+3$};
\draw [fill=black] (7.649066664678466,-3.0358185854701913) circle (2.5pt);
\draw[color=black] (-2.9898556318664222,9.83350665977712) node {$2(k+\ell)+1$};
\draw [fill=black] (7.9095389311788615,-1.5586069559518803) circle (2.5pt);
\draw[color=black] (-3.372124030071655,9.83350665977712) node {$2k+8$};
\draw [fill=black] (7.159538931178862,-0.2595688502752225) circle (2.5pt);
\draw[color=black] (7.367797633789652,-0.032658665329290265) node {$2k+7$};
\draw [fill=black] (5.75,0.25346136471328107) circle (2.5pt);
\draw[color=black] (5.438252385706096,0.49523578933507484) node {$2k+6$};
\draw [fill=black] (4.3404610688211385,-0.25956885027522114) circle (2.5pt);
\draw[color=black] (4.073008106401692,-0.05086192238668216) node {$2k+5$};
\draw [fill=black] (3.590461068821138,-1.5586069559518787) circle (2.5pt);
\draw[color=black] (-3.372124030071655,9.83350665977712) node {$2k+4$};
\draw [fill=black] (3.850933335321533,-3.0358185854701905) circle (2.5pt);
\draw[color=black] (-3.372124030071655,9.83350665977712) node {$2k+3$};
\draw [fill=black] (-2.,-4.) circle (2.5pt);
\draw[color=black] (-3.5359533435881834,9.83350665977712) node {$3$};
\draw [fill=black] (-0.5,-4.) circle (2.5pt);
\draw[color=black] (-3.372124030071655,9.83350665977712) node {$2k+1$};
\draw [fill=black] (0.6490666646784662,-3.0358185854701913) circle (2.5pt);
\draw[color=black] (-3.4813435724160073,9.83350665977712) node {$2k$};
\draw [fill=black] (0.909538931178862,-1.5586069559518803) circle (2.5pt);
\draw[color=black] (-3.4267338012438313,9.83350665977712) node {$p+1$};
\draw [fill=black] (0.15953893117886264,-0.2595688502752225) circle (2.5pt);
\draw[color=black] (0.195714353177186,-0.032658665329290265) node {$p$};
\draw [fill=black] (-1.25,0.25346136471328107) circle (2.5pt);
\draw[color=black] (-1.5517983243324505,0.49523578933507484) node {$2k+2$};
\draw [fill=black] (-2.6595389311788615,-0.25956885027522114) circle (2.5pt);
\draw[color=black] (-2.9898556318664222,-0.10547169355885785) node {$1$};
\draw [fill=black] (-3.409538931178862,-1.5586069559518787) circle (2.5pt);
\draw[color=black] (-3.4267338012438313,9.83350665977712) node {$p-1$};
\draw [fill=black] (-3.149066664678467,-3.0358185854701905) circle (2.5pt);
\draw[color=black] (-3.5359533435881834,9.83350665977712) node {$4$};
\end{scriptsize}
\end{tikzpicture}
\caption{Minimum coprime labeling of $C_{2k+1}\cup C_{2\ell+1}$ with $p=3$}\label{case1}
\end{center}
\end{figure}
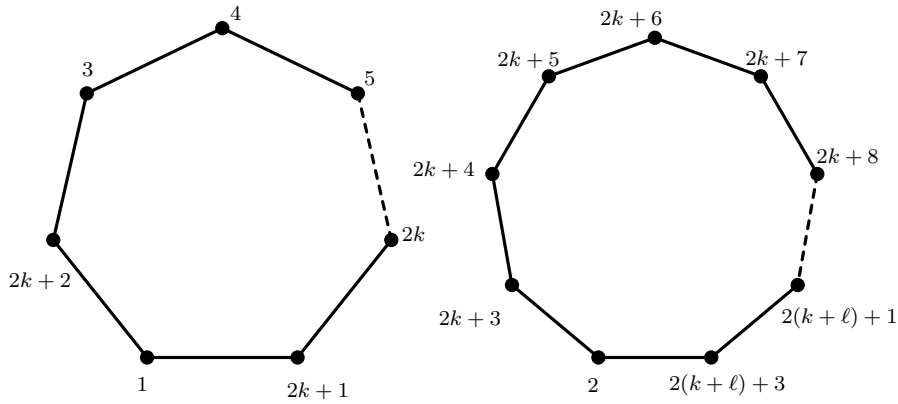

Next, we assume that the smallest prime $p$ that does not divide $2k+2$ is greater than $3$. Refer to Figure~\ref{case2} for the coprime labeling of this case.  Note the labeling of $C_{2\ell+1}$ is identical to the previous case.  In addition to edges connecting consecutive labels or labels adjacent to $1$, we observe that $3$ and $2k+1$ are coprime based on our assumption that $3\mid (2k+2)$.  Furthermore, the prime $p$ is chosen to be coprime with its adjacent label $2k+2$.  Therefore, this labeling is a coprime labeling of $C_{2k+1}\cup C_{2\ell+1}$ for our second and final case.

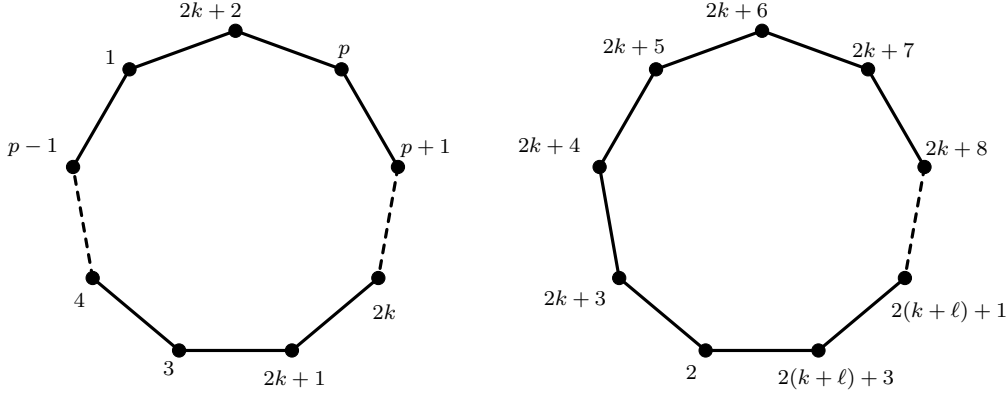
\begin{figure}[htb]
\begin{center}
\begin{tikzpicture}[line cap=round,line join=round,>=triangle 45,x=1.0cm,y=1.0cm]
\clip(-4.402825383888248,-4.813335078262535) rectangle (9.72810302323581,0.8504783318802561);
\draw [line width=1.2pt] (1.,3.)-- (-1.,3.);
\draw [line width=1.2pt] (-1.,3.)-- (-2.2469796037174667,4.56366296493606);
\draw [line width=1.2pt] (-2.2469796037174667,4.56366296493606)-- (-1.8019377358048378,6.513518789299707);
\draw [line width=1.2pt] (-1.8019377358048378,6.513518789299707)-- (0.,7.381286267534822);
\draw [line width=1.2pt] (0.,7.381286267534822)-- (1.8019377358048383,6.513518789299706);
\draw [line width=1.2pt,dash pattern=on 3pt off 3pt] (1.8019377358048383,6.513518789299706)-- (2.2469796037174667,4.563662964936059);
\draw [line width=1.2pt] (2.2469796037174667,4.563662964936059)-- (1.,3.);
\draw [line width=1.2pt] (6.5,3.)-- (5.,3.);
\draw [line width=1.2pt] (5.,3.)-- (3.850933335321533,3.9641814145298095);
\draw [line width=1.2pt] (3.850933335321533,3.9641814145298095)-- (3.5904610688211385,5.441393044048121);
\draw [line width=1.2pt] (3.5904610688211385,5.441393044048121)-- (4.3404610688211385,6.740431149724778);
\draw [line width=1.2pt] (4.3404610688211385,6.740431149724778)-- (5.75,7.25346136471328);
\draw [line width=1.2pt] (5.75,7.25346136471328)-- (7.159538931178862,6.740431149724777);
\draw [line width=1.2pt] (7.159538931178862,6.740431149724777)-- (7.9095389311788615,5.4413930440481195);
\draw [line width=1.2pt,dash pattern=on 3pt off 3pt] (7.9095389311788615,5.4413930440481195)-- (7.6490666646784655,3.9641814145298078);
\draw [line width=1.2pt] (7.6490666646784655,3.9641814145298078)-- (6.5,3.);
\draw [line width=1.2pt] (6.5,-4.)-- (5.,-4.);
\draw [line width=1.2pt] (5.,-4.)-- (3.850933335321533,-3.0358185854701905);
\draw [line width=1.2pt] (3.850933335321533,-3.0358185854701905)-- (3.590461068821138,-1.5586069559518787);
\draw [line width=1.2pt] (3.590461068821138,-1.5586069559518787)-- (4.3404610688211385,-0.25956885027522114);
\draw [line width=1.2pt] (4.3404610688211385,-0.25956885027522114)-- (5.75,0.25346136471328107);
\draw [line width=1.2pt] (5.75,0.25346136471328107)-- (7.159538931178862,-0.2595688502752225);
\draw [line width=1.2pt] (7.159538931178862,-0.2595688502752225)-- (7.9095389311788615,-1.5586069559518803);
\draw [line width=1.2pt,dash pattern=on 3pt off 3pt] (7.9095389311788615,-1.5586069559518803)-- (7.649066664678466,-3.0358185854701913);
\draw [line width=1.2pt] (7.649066664678466,-3.0358185854701913)-- (6.5,-4.);
\draw [line width=1.2pt] (-0.5,-4.)-- (-2.,-4.);
\draw [line width=1.2pt] (-2.,-4.)-- (-3.149066664678467,-3.0358185854701905);
\draw [line width=1.2pt,dash pattern=on 3pt off 3pt] (-3.149066664678467,-3.0358185854701905)-- (-3.409538931178862,-1.5586069559518787);
\draw [line width=1.2pt] (-3.409538931178862,-1.5586069559518787)-- (-2.6595389311788615,-0.25956885027522114);
\draw [line width=1.2pt] (-2.6595389311788615,-0.25956885027522114)-- (-1.25,0.25346136471328107);
\draw [line width=1.2pt] (-1.25,0.25346136471328107)-- (0.15953893117886264,-0.2595688502752225);
\draw [line width=1.2pt] (0.15953893117886264,-0.2595688502752225)-- (0.909538931178862,-1.5586069559518803);
\draw [line width=1.2pt,dash pattern=on 3pt off 3pt] (0.909538931178862,-1.5586069559518803)-- (0.6490666646784662,-3.0358185854701913);
\draw [line width=1.2pt] (0.6490666646784662,-3.0358185854701913)-- (-0.5,-4.);
\begin{scriptsize}
\draw [fill=black] (-1.,3.) circle (2.5pt);
\draw[color=black] (-1.2753324570753661,2.633530701740024) node {$1$};
\draw [fill=black] (1.,3.) circle (2.5pt);
\draw[color=black] (1.2800581050766229,2.538180307629876) node {$2k+1$};
\draw [fill=black] (2.2469796037174667,4.563662964936059) circle (2.5pt);
\draw[color=black] (4.1977801648471775,3.7777354310617994) node {$2k$};
\draw [fill=black] (1.8019377358048383,6.513518789299706) circle (2.5pt);
\draw[color=black] (4.846162844796189,3.7968055098838294) node {$5$};
\draw [fill=black] (0.,7.381286267534822) circle (2.5pt);
\draw[color=black] (-1.7139442699820508,3.7968055098838294) node {$4$};
\draw [fill=black] (-1.8019377358048378,6.513518789299707) circle (2.5pt);
\draw[color=black] (1.089357316856325,3.7968055098838294) node {$3$};
\draw [fill=black] (-2.2469796037174667,4.56366296493606) circle (2.5pt);
\draw[color=black] (-1.2562623782533362,3.567964564019474) node {$2k+2$};
\draw [fill=black] (5.,3.) circle (2.5pt);
\draw[color=black] (4.71267229304198,2.6144606229179943) node {$2$};
\draw [fill=black] (6.5,3.) circle (2.5pt);
\draw[color=black] (6.734100648177137,2.519110228807846) node {$2(k+\ell)+3$};
\draw [fill=black] (7.6490666646784655,3.9641814145298078) circle (2.5pt);
\draw[color=black] (8.240636875117488,3.5107543275533852) node {$2(k+\ell)+1$};
\draw [fill=black] (7.9095389311788615,5.4413930440481195) circle (2.5pt);
\draw[color=black] (9.556472313837542,3.7395952734177405) node {$2k+8$};
\draw [fill=black] (7.159538931178862,6.740431149724777) circle (2.5pt);
\draw[color=black] (5.4564053671011425,3.7395952734177405) node {$2k+7$};
\draw [fill=black] (5.75,7.25346136471328) circle (2.5pt);
\draw[color=black] (3.5112573272541057,3.7395952734177405) node {$2k+6$};
\draw [fill=black] (4.3404610688211385,6.740431149724778) circle (2.5pt);
\draw[color=black] (2.1382116520679624,3.7395952734177405) node {$2k+5$};
\draw [fill=black] (3.5904610688211385,5.441393044048121) circle (2.5pt);
\draw[color=black] (5.208494342414755,3.7395952734177405) node {$2k+4$};
\draw [fill=black] (3.850933335321533,3.9641814145298095) circle (2.5pt);
\draw[color=black] (3.2633463025677187,3.4916842487313557) node {$2k+3$};
\draw [fill=black] (5.,-4.) circle (2.5pt);
\draw[color=black] (4.81267229304198,-4.284258304766869) node {$2$};
\draw [fill=black] (6.5,-4.) circle (2.5pt);
\draw[color=black] (6.734100648177137,-4.3796086988770165) node {$2(k+\ell)+3$};
\draw [fill=black] (7.649066664678466,-3.0358185854701913) circle (2.5pt);
\draw[color=black] (8.240636875117488,-3.4879646001314772) node {$2(k+\ell)+1$};
\draw [fill=black] (7.9095389311788615,-1.5586069559518803) circle (2.5pt);
\draw[color=black] (8.335987269227637,-1.3139756144201034) node {$2k+8$};
\draw [fill=black] (7.159538931178862,-0.2595688502752225) circle (2.5pt);
\draw[color=black] (7.363413249304119,-0.0172102545220906) node {$2k+7$};
\draw [fill=black] (5.75,0.25346136471328107) circle (2.5pt);
\draw[color=black] (5.418265209457083,0.5167519524947382) node {$2k+6$};
\draw [fill=black] (4.3404610688211385,-0.25956885027522114) circle (2.5pt);
\draw[color=black] (4.045219534270939,0.0362803333441202) node {$2k+5$};
\draw [fill=black] (3.590461068821138,-1.5586069559518787) circle (2.5pt);
\draw[color=black] (2.9200848837711826,-1.2758354567760442) node {$2k+4$};
\draw [fill=black] (3.850933335321533,-3.0358185854701905) circle (2.5pt);
\draw[color=black] (3.2633463025677187,-3.3070346789535067) node {$2k+3$};
\draw [fill=black] (-2.,-4.) circle (2.5pt);
\draw[color=black] (-2.1360466346429438,-4.234258304766869) node {$3$};
\draw [fill=black] (-0.5,-4.) circle (2.5pt);
\draw[color=black] (-0.4650899347704134,-4.3796086988770165) node {$2k+1$};
\draw [fill=black] (0.6490666646784662,-3.0358185854701913) circle (2.5pt);
\draw[color=black] (0.7270258192377594,-3.4879646001314772) node {$2k$};
\draw [fill=black] (0.909538931178862,-1.5586069559518803) circle (2.5pt);
\draw[color=black] (1.2800581050766229,-1.3139756144201034) node {$p+1$};
\draw [fill=black] (0.15953893117886264,-0.2595688502752225) circle (2.5pt);
\draw[color=black] (0.193063612220926,-0.0172102545220906) node {$p$};
\draw [fill=black] (-1.25,0.25346136471328107) circle (2.5pt);
\draw[color=black] (-1.5804537182278424,0.5167519524947382) node {$2k+2$};
\draw [fill=black] (-2.6595389311788615,-0.25956885027522114) circle (2.5pt);
\draw[color=black] (-2.9107096298800746,-0.093490569810209) node {$1$};
\draw [fill=black] (-3.409538931178862,-1.5586069559518787) circle (2.5pt);
\draw[color=black] (-3.940493886269682,-1.2949055355980736) node {$p-1$};
\draw [fill=black] (-3.149066664678467,-3.0358185854701905) circle (2.5pt);
\draw[color=black] (-3.3256017580748784,-3.3354039695552404) node {$4$};
\end{scriptsize}
\end{tikzpicture}
\caption{Minimum coprime labeling of $C_{2k+1}\cup C_{2\ell+1}$ with $p>3$}\label{case2}
\end{center}
\end{figure}

It remains to show that this labeling is minimum, which is accomplished by demonstrating that the graph does not have a prime labeling since $2(k+\ell)+3$ is one larger than the size of the vertex set~$V$ of $C_{2k+1}\cup C_{2\ell+1}$.  No prime labeling exists because the independence number for the graph is $k+\ell$, which is smaller than the required $\dfrac{|V|}{2}$ of independent vertices needed for the even labels of a prime labeling (a fact first noted in \cite{FH}). Therefore, our labeling is a minimum coprime labeling.

\end{proof}

We now find a minimum coprime labeling for the union of the complete graph with a path or a star graph.  These classes of graphs, $K_m\cup P_n$ and $K_m \cup S_n$, were investigated by Youssef and El Sakhawi in~\cite{EY}. They proved that $K_m\cup P_n$ is prime if and only if $1\leq m\leq 3$ or $m=4$ and $n\geq 1$ is odd. See Figure~\ref{completePath} for an example of such a graph with a minimum coprime labeling.

\begin{figure}[htb]
\begin{center}
\includegraphics[scale=1.2]{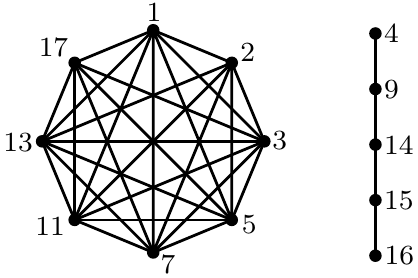}
\end{center}
\caption{Minimum coprime labeling of $K_8\cup P_5$ when $n\leq r$}\label{completePath}
\end{figure}

\begin{theorem}\label{unionCompletePath}
Let $m$ and $n$ be positive integers and let $r=2\left(\left\lfloor\frac{p_{m-1}}{2}\right\rfloor-m+2\right)$. The minimum coprime labeling for $K_m\cup P_n$ is 
\[
\pr(K_m\cup P_n) =\begin{cases}
m+n &\mbox{if $1\leq m\leq 3$ or $m=4$ and $n$ is odd,} \\
m+n+1 & \mbox{if $m=4$ and $n\geq 2$ is even,}\\
p_{m-1} &\mbox{if $m\geq 5$ and $n\leq r$}\\
p_{m-1}+ n-r+1 &\mbox{if $m\geq 5$ and $n> r$.}\\
\end{cases}
\]
\end{theorem}

\begin{proof}
The case when $\pr(K_m\cup P_n)=m+n$ was already shown in~\cite{EY} since the union graph is prime under those conditions. 
When $m=4$ and $n$ is even, we can label $K_4$ with $1,2,3,5$ and use the integers $6,7,\ldots,n+5$ to label $P_n$, so $\pr(K_m\cup P_n)=m+n+1$ since the graph is not prime. 
Suppose that $n\leq r$. It is clear by our discussion on $K_m$ that $\pr(K_m\cup P_n)\geq p_{m-1}$ since $K_m$ is a subgraph of $K_m\cup P_n$. Label the vertices of $K_m$ with $1$ and the first $n-1$ primes. If $n\leq 5$, use the sequence $4,9,8,15,14$ to label $P_n$.  Otherwise, we label the vertices on the path $P_n$ with the following sequence, in which $m'$ denotes the largest label of the sequence:
\[
8,9,14,15,16,21,22,25,\ldots,m',4\;\;\text{if $n$ is odd or}\;\; 8,9,14,15,16,21,22,25,\ldots,4,m'\;\; \text{if $n$ is even}
\]
where we skip each prime $p_i$ and the next even integer $p_i+1$ starting with $p_9=23$.
There are $\left\lfloor \frac{p_{m-1}}{2}\right\rfloor$ odd composite numbers between $0$ and $p_{m-1}$ (including $1$). Of these, $m-2$ of them are odd prime numbers. So we can use at most $2\left(\left\lfloor\frac{p_{m-1}}{2}\right\rfloor-m+2\right)$ integers for the labels on $P_n$ that are all less than $p_{m-1}$. 
If $p_i$ is not a twin prime and $\gcd(p_i-1,p_i+2)=3$, then replace $p_i-1$ in the sequence with $p_i+1$. Since $p_i+2$ is divisible by $3$, $\gcd(p_i+1,p_i-2)=1$. If $p_i$ is a twin prime with $p_{i+1}$ and $\gcd(p_i-1,p_i+4)=5$, then replace $p_i-1$ in the sequence with $p_i+3$. Since $p_i+4$ is divisible by $5$, $\gcd(p_i+3,p_i-2)=1$. It is clear that each pair of adjacent labels in $K_m$ are relatively prime. Since the distance between any pair of labels in $P_n$ is at most $5$ and we ensure we make a switch if a pair of adjacent labels are divisible by $3$ or $5$, $\pr(K_m\cup P_n)$ is less than or equal to the indicated values above. Thus $\pr(K_m\cup P_n)=p_{m-1}$. 

Suppose that $n> r$. We label the first $r$ vertices as above. The vertices that come after $r$ on the path $P_n$ have the following sequence:
\[
 p_{m-1}+2,p_{m-1}+3,p_{m-1}+4,\ldots
\]
Thus $\pr(K_m\cup P_n)\leq p_{m-1}+ n-r+1$. To show $\pr(K_m\cup P_n)$ is bounded below by the same quantity, we must first recall that the path has at most $\left\lfloor \frac{n-1}{2}\right\rfloor$ odd numbers. By Proposition~\ref{complete}, we require at least $p_{m-1}$ integers to label $K_n$. There are $r$ integers we can place on the path $P_n$ ending on the $r^{\rm th}$ vertex which is always an even label. 
Our next label must be $p_{m-1}+2$, and so it follows that by having at most $\left\lfloor \frac{n-1}{2}\right\rfloor$ odd numbers, $\pr(K_m\cup P_n)=p_{m-1}+n-r+1$.
\end{proof}

The last disjoint union graph we investigate is between a complete graph and a star. See Figure~\ref{completeStar} for an example of the union of $K_8$ and the star $S_6$, with a minimum coprime labeling. To find the minimum coprime labeling of such a graph, we first must find the number of composite numbers not equal to $1$ that are less than or equal to a number $b$ and relatively prime with another number $t$; we denote such a quantity as $\varphi(t,b)$. 
For example, $\varphi(25,31)=14$ since there are $11$ composite numbers less than 25 that are relatively prime with $25$ and $3$ such numbers that are between $25$ and $31$. 
Next we will label the center of the star with an integer $k_b=\min(\{q^2\,:\, q^2<b,\varphi(q^2,b)\geq n, \text{ and $q$ is prime}\})$ for some composite integer $b\in[p_{m-1},p_m]$. For example when $n=10$, $k_{27}=25$ but $k_{28}=9$. Note that $k_b$ may not be defined for some values of $b$, but as long as there are enough composite integers less than $b$ that are relatively prime with $k_b$, then there will be some $k_b$ that is defined. 
For those $k_b$ that are not defined, let them be defined as $\infty$. The integer $b$ will represent the largest number among the pendant vertices of the star in $K_m\cup S_n$. Thus we pick the center of the star to be labeled with the integer $k_b$ such that we minimize the largest label among the vertices of $S_n$, which would be either $k_b$ or $b$, i.e., let $\alpha=\min(\{\max(k_b,b)\,:\,b\in[p_{m-1},p_m] \text{ and $b$ is composite}\})$.
We formalize this in the following result.

\begin{figure}[htb]
\begin{center}
\includegraphics[scale=1.1]{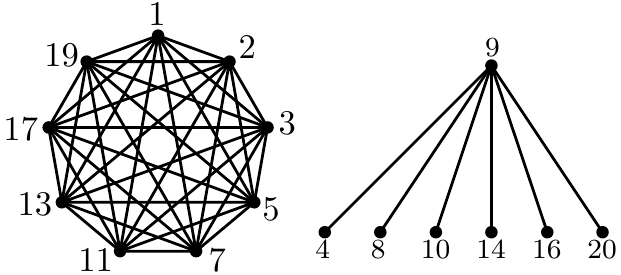}
\end{center}
\caption{Minimum coprime labeling of $K_8\cup S_6$ with $\pi(m+n+1)<m$, $n>r$, and $\alpha\leq p_m$}\label{completeStar}
\end{figure}
\begin{theorem}\label{unionCompleteStar}
Let $m,n$ be positive integers, $p$ be the largest prime number such that $p^2<p_{m-1}$, and $r=p_{m-1}-\left\lfloor \frac{p_{m-1}}{p}\right\rfloor -m+1$. 
The minimum coprime number for $K_m\cup S_n$ where $S_n$ is a star with $n$ pendant vertices is 
\[
\pr(K_m\cup S_n) = \begin{cases}
m+n+1 &\mbox{if $\pi(m+n+1)\geq m$,} \\
p_{m-1} &\mbox{if $\pi(m+n+1)< m$ and $n\leq r$} \\
\alpha &\mbox{if $\pi(m+n+1)<m$, $n>r$, and $\alpha \leq p_m$} \\
p_m & \mbox{otherwise}.\\ 
\end{cases}
\]
\end{theorem}

\begin{proof}
By the work in \cite{EY}, if $\pi(m+n+1)\geq m$ then $\pr(K_m\cup S_n)=m+n+1$. Now suppose that $\pi(m+n+1)<m$. Since $\pr(K_m)=p_{m-1}$ by Proposition~\ref{complete}, $\pr(K_m\cup S_n)\geq p_{m-1}$. 
It is clear that $K_m\cup S_n$ will not be prime in this case. We first suppose that $n\leq r$. 
Again, label the vertices of $K_m$ using $1$ and the first $m-1$ primes. Let the center of the star be $p^2$. There are $p_{m-1}-m$ positive integers that are not used on the labels of $K_m$. Depending on the center of the star, we will not be able to use some of these positive integers for the labels of $S_n$. 
Since the center of the star is labeled $p^2$, there are at most $\left\lfloor \frac{p_{m-1}}{p}\right\rfloor-1$ positive integers less than $p_{m-1}$ that we can use to label the pendant vertices of $S_n$. So as long as $n\leq p_{m-1}-m-\left\lfloor\frac{p_{m-1}}{p}\right\rfloor +1=r$, $\pr(K_m\cup S_n)=p_{m-1}$. 

Now suppose that $n>r$. Then we can use at most $p_m-p_{m-1}$ additional positive integers on the pendant vertices of $S_n$ before it would be better to use $m$ primes on $K_m$ and letting $1$ be the center of the star, in which case $\pr(K_m\cup S_n)=p_m$. 
Since there are many options for the center of the star, we aim to find the smallest composite number less than $p_m$ that is relatively prime with at least $n$ composite numbers less than some fixed value $b\in[p_{m-1},p_m]$. Since $q^2$ ($q$ prime) is relatively prime with at least the same number of composite numbers as $qt$ for some integer $t$, we would choose $q^2$ over $qt$ at all times. Thus $k_b$ is defined to be a candidate for the center of the star for each composite number $b\in[p_{m-1},p_m]$. If this were a minimum coprime labeling then its minimum coprime number is $\max(k_b,b)$. 
Hence we look for the minimum among all of these maximums and denote such a value as~$\alpha$. By construction, there are more than $n$ relatively prime composite numbers less than $\alpha$ that are not~$1$, so the pendant vertices of the star can be labeled. Thus, this is a coprime labeling.
By construction, we choose $\alpha$ to be minimum and thus the result follows. 
\end{proof}

Next we consider the corona of a complete graph with the empty graph of one or two vertices.  In \cite{EY}, it was shown that $K_n\odot K_1$ and $K_n\odot\overline{K}_2$ are prime under certain conditions, particularly if $n\leq 7$ for $K_1$ and if $n\leq 16 $ for $\overline{K}_2$. Later, El Sonbaty, Mahran, and Seoud \cite{EMS} showed that $K_n\odot \overline{K}_m$ is not prime if $n>\pi(n(m+1))+1$ and conjectured that $K_n\odot \overline{K}_m$ is prime if $n<\pi(n(m+1))+1$.
We give the minimum coprime number for $K_n\odot K_1$ and $K_n\odot \overline{K}_2$ below. An example of $K_n\odot K_1$ is given in Figure~\ref{completeK1}.

\begin{figure}[htb]
\begin{center}
	\includegraphics[scale=1]{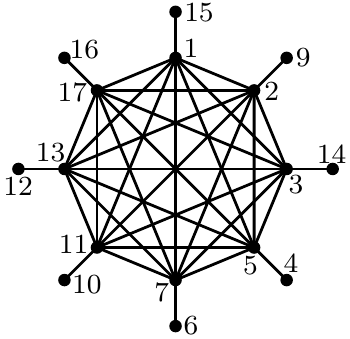}
\end{center}
\caption{Minimum coprime labeling of $K_8\odot K_1$}\label{completeK1}
\end{figure}

\begin{theorem}
Let $n$ be a positive integer. $\pr(K_n\odot K_1)= p_{n-1}$ for $n>7$.
\end{theorem}

\begin{proof}
By Proposition~\ref{complete} and Observation~\ref{primeFromPrime}, $\pr(K_n\odot K_1)\geq p_{n-1}>16$.  Let $u_1,u_2,\ldots,u_n$ be the vertices in $K_n$ and for each $i\in\{1,\ldots,n\}$, let $v_{i}$ be the vertices adjacent to $u_i$ for the $n$ copies of $K_1$. Label the vertices $u_{i}$ with $p_{i-1}$ for $i\in\{2,\ldots,n-1\}$ and label $u_1$ with $1$. Then label $v_i$ with $p_{i}-1$ for $i\in\{4,\ldots,n\}$. 
We label $v_1,v_2,v_3$ with $15,9,14$ respectively. Thus, our result follows since this is a minimum coprime labeling with largest label being $p_{n-1}$.
\end{proof}

\begin{theorem}
If $n>16$ then $\pr(K_n\odot \overline{K}_2)=p_{n-1}$.
\end{theorem}

\begin{proof}
By Proposition~\ref{complete}, $\pr(K_n\odot K_1)\geq p_{n-1}\geq 53$. 
We will label the vertices in $K_n$ and the first~16 corresponding copies of $\overline{K}_2$, and then we will label the remaining vertices in a more structured manner. Let $u_1,u_2,\ldots,u_n$ be the vertices in $K_n$ and $v_{i,1},v_{i,2}$ be the vertices for the $n$ copies of $\overline{K}_2$. Label the vertices $u_i$ for $i\in\{2,3,\ldots,n\}$ with $p_{i-1}$ and label $u_1$ with $1$. We label the sequence of vertices $v_{1,1},v_{1,2},v_{2,1},v_{2,2},v_{3,1},v_{3,2},\ldots,v_{16,2}$ respectively with the following sequence of labels:
\begin{align*}
&4,6,9,15,8,14,12,16,10,18,20,21,22,24,25,26,27,28,\\
&30,32,33,34,35,36,38,39,40,42,44,45,46,48.
\end{align*}
For $i>16$, if $p_{i-1}$ is not the first of a pair of twin primes, we label $v_{i,1},v_{i,2}$ with $p_{i-1}-2,p_{i-1}-1$. For $i>16$, if $p_{i-1}$ is the first of a pair of twin prime vertices, we label $v_{i,1},v_{i,2}$ with $p_{i-1}-3,p_{i-1}-2$ and label $v_{i+1,1},v_{i+1,2}$ with $p_{i-1}-1,p_{i-1}+1$. Therefore, we have obtained our result since this is clearly a minimum coprime labeling with largest label being $p_{n-1}$.
\end{proof}

As long as there are enough integers less than $p_{n-1}$ the authors believe that the minimum prime labeling of $K_n\odot \overline{K}_m$ is $p_{m-1}$. It is likely that several of the small $m$ can be easily done as long as enough care is taken with the labeling of the first several pendant vertices, but a generalization eludes discovery. As such, we leave this as an open problem in Section~\ref{remarks} and a conjecture below.

\begin{conjecture}
For all $m>0$, there exists an $M>m$ such that for all $n>M$, $\pr(K_m\odot\overline{K}_m)=p_{n-1}$.
\end{conjecture}

\section{Powers of Paths and Cycles}\label{powers}

We consider $P_n^2$ to have $n\geq 2$ vertices. Seoud and Youssef~\cite{SY} proved that this graph is not prime when $n=6$ and $n\geq 8$.  We will construct a minimum coprime labeling of $P_n^2$ for these cases.  A lower bound for the minimum coprime number for the graph would be obtained by using the maximum amount of even labels that can be used based on the independence number of $P_n^2$, shown in~\cite{SY} to be $\displaystyle \left\lceil\frac{n}{3}\right\rceil$, along with the smallest possible odd labels.
Figure~\ref{pathsquared} shows a minimum coprime labeling of the graph $P_{6}^2$ and $P_{10}^2$, and in the former case, the path is represented as $v_1,v_2,\ldots,v_6$ with the horizontal edges connecting vertices of distance $2$.  Since the independence number of $P_6^2$ is $2$, we can only use two even labels, which prevents a prime labeling.  Instead, a minimum coprime labeling can be achieved with $\pr(P_6^2)=7$.

\begin{figure}[htb]
\begin{center}
\includegraphics[scale=1]{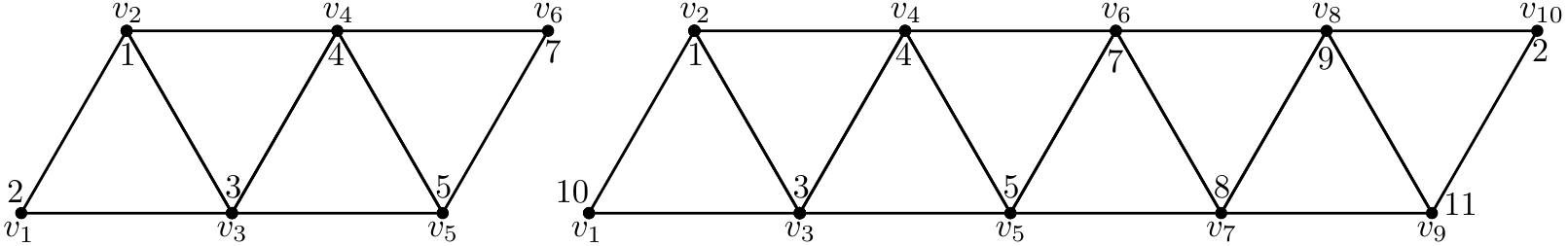}
\caption{Minimum coprime labeling of $P_6^2$(left) and $P_{10}^2$ (right)}\label{pathsquared}
\end{center}
\end{figure}

The following theorem regarding the minimum coprime number of the path squared was verified for the $n=6$ through the labeling in Figure~\ref{pathsquared}, and the general case of $n\geq 8$ will be proven by a series of lemmas. 

\begin{theorem}\label{pr_path_squared}
Let $n=6$ or $n\geq 8$.  The minimum coprime number of $P_n^2$ is given by 
$$\pr(P_n^2)=\begin{cases}
4k-1 & \text{if } n=3k \text{ or }3k+1\\
4k+1 & \text{if } n=3k+2.\\
\end{cases}$$
\end{theorem}

Assume that $n\geq 8$ for the following lemmas that will prove the general case of Theorem~\ref{pr_path_squared}.  To construct a minimum coprime labeling of $P_n^2$, we define a sequence $X=\{x_i\}_{i=1}^{\infty}$ of integers for which the first $n$ will be used as labels for the vertices $\{v_1,\ldots, v_n\}$.  The sequence $X$ consists of a length $45$ segment with the repeated pattern of even, odd, and odd integers, with the subsequent terms of the sequence defined by shifting the entries by $60$.  The first $30$ odd numbers are included in the initial segment along with the first $15$ even numbers that are not multiples of $3$ or $5$.  The definition of the sequence is the following:
\begin{align*}
\{x_1,\ldots,x_{45}\}=\{&2,1,3,4,5,7,8,9,11,14,13,15,16,17,19,22,21,23,26,25,27,28,29,31,32,33,35,\\
&34,37,39,38,41,43,44,45,47,46,49,51,52,53,55,56,57,59\}
\end{align*} and for $i> 45$, $x_i=x_{i-45}+60$.

In order to examine whether adjacent vertices will have relatively prime labels, the following fact about the distance between the labels of such vertices will be quite useful.

\begin{lemma}\label{distanceLemma1}
Given adjacent vertices $v_i$ and $v_j$ in $P_n^2$ for some $1\leq i,j\leq n$, the labels in the sequence $\{x_1,\ldots,x_n\}$ satisfy $|x_i-x_j|\leq 5$.
\end{lemma}
\begin{proof}
Based on the structure of $P_n^2$, the neighborhood of a vertex $v_i$ with $i=3,\ldots,n-2$ is the set $\{v_{i-2},v_{i-1},v_{i+1},v_{i+2}\}$.  By inspection of the $45$ labels in the initial segment of the sequence, we can verify that $x_i$ is within $5$ of the label of any adjacent vertex when $i\leq 43$.  

For adjacent vertices with labels that lie in different length $45$ segments, we have $x_{44}=57$, $x_{45}=59$, $x_{46}=62$, and $x_{47}=61$, where the latter two labels are the result of shifting $x_1$ and $x_2$ by~$60$.  We see that $|x_{44}-x_{46}|=5$ is the furthest apart labels for adjacent pairs of these vertices will be, which still satisfies our desired inequality.

Adjacent vertices with indices that are both larger than $45$ will still satisfy $|x_i-x_j|\leq 5$ because $x_i=x_{a}+60m$ and $x_j=x_{b}+60m$ for some integers $1\leq a,b\leq 47$ and positive integer $m$.  Since $x_a$ and $x_b$ satisfy the inequality as shown above, the shifted values $x_i$ and $x_j$ also maintain a distance of $5$ or less, which covers the remaining possible cases of indices $i$ and $j$.
\end{proof}

Note that if we continued to define $x_{46},\ldots,x_{54}$ in the manner of the first $45$ terms by including all even numbers that are not multiples of $3$ or $5$, then we would have $x_{52}=64$ and $x_{54}=71$.  This would have contradicted Lemma~\ref{distanceLemma1} and eventually would result in the sequence containing adjacent labels that are both multiples of $7$.  Shifting each length $45$ segment by $60$ to have $x_{46}=62$ skips the even number $58$ and maintains that the distance between adjacent labels remain small enough to guarantee a coprime labeling, as shown in the following lemma.
 
\begin{lemma}\label{coprimeLabel1}
The labels $\{x_1,\ldots,x_n\}$ are a coprime labeling of $P_n^2$.
\end{lemma}

\begin{proof}
By Lemma~\ref{distanceLemma1}, the distance between the labels of adjacent vertices is at most $5$; hence, no two adjacent vertices will have labels with a factor of $6$ or higher in common. By design, even labels are not adjacent, eliminating the possibility of labels sharing a factor of $2$. None of the even labels are divisible by $3$ or $5$, and any two odd labels both divisible by $3$ or $5$ are spaced out enough to avoid being adjacent. Thus our labeling is coprime. 
\end{proof}

\begin{lemma}\label{minimumCoprime}
If $\max{(x_1,\ldots,x_n)}$ is odd, then the labels $\{x_1,\ldots,x_n\}$ are a minimum coprime labeling of $P_n^2$.
\end{lemma}

\begin{proof}
Notice that even integer labels cannot be 1 or 2 indicies apart, which corresponds to the independence number of our graph being $\displaystyle \left\lceil\frac{n}{3}\right\rceil$. Hence, we have used as few odd labels as possible, while also using all of the odd numbers below $\max(x_1,\ldots,x_n)$. Thus we cannot make $\pr(P_n^2)$ any smaller.
\end{proof}

For the case of the maximum label within $\{x_1,\ldots, x_n\}$ being even, we alter the sequence of labels to achieve a minimum labeling.  Through examination of the first $45$ terms of the sequence and the fact that remaining terms are simply shifted from this initial segment, we see that this situation can only occur when $n=3k+1$ or $3k+2$.  Additionally, we observe that the largest even label in either case is in position $3k+1$.  We create a new sequence $\{x_i^*\}_{i=1}^n$ by defining $x_1^*=10$, $x_{3k+1}^*=2$, and $x_i^*=x_i$ for $i\in\{2,\ldots,n\}\setminus\{3k+1\}$.

\begin{lemma}\label{alteredMinimumCoprime}
If $\max{(x_1,\ldots,x_n)}$ is even, then the labels $\{x_1^*,\ldots,x_n^*\}$ are a minimum coprime labeling of $P_n^2$.
\end{lemma}

\begin{proof}
Note that by inspection of the initial segment of our sequence of labels, the first even value larger than all preceding values of $n$ (assuming $n\geq 8$) would be the $10^{\rm th}$ element in which $x_{10}=14$.  For this case, or for any larger $n$ value, replacing the first and $(3k+1)^{\rm st}$ entries in the sequence with $10$ and~$2$ respectively would result in the maximum of the labels falling into the case of Lemma~\ref{minimumCoprime} since one can observe that the second largest label will always be odd.  

Following the reasoning of the previous lemma, the labeling $\{x_1^*,\ldots,x_n^*\}$ is a minimum coprime labeling as long as we show it is still a coprime labeling.  This sequence matches the original sequence except for two values; hence, the only adjacent vertex pairs that need to be checked as still having relatively prime labels are the vertices $v_1$ and $v_{3k+1}$ with their respective neighborhoods.  In the case of $n=3k+1$, $v_{3k+1}=v_n$ is only adjacent to $v_{n-1}$ and $v_{n-2}$.  Since both labels for these vertices are odd, the label $x_n^*=2$ is relatively prime with its adjacent labels.  Similarly if $n=3k+2$, then $v_{3k+1}=v_{n-1}$ is only adjacent to $v_{n-3}$, $v_{n-2}$, and $v_{n}$.  Again, all of these vertices will be labeled by odd numbers, which are relatively prime to $2$.  Likewise, the label $x_1^*=10$ is only adjacent to the second and third vertices with labels $x_2^*=1$ and $x_3^*=3$, so the labels are once again relatively prime.  Thus, the labeling $\{x_1^*,\ldots,x_n^*\}$ is a minimum coprime labeling.
\end{proof}

\begin{proof}[Proof of Theorem~\ref{pr_path_squared}.]
The previous lemmas in this section have shown that the sequence $\{x_1,\ldots,x_n\}$ or $\{x_1^*,\ldots,x_n^*\}$  provides a minimum coprime labeling of $P_n^2$.  It only remains to show that the maximum value in this sequence, which was shown to be odd, is in fact $4k-1$ or $4k+1$ depending on the value of $n\pmod{3}$.  In the case of $n=3k$, there are $\displaystyle \left\lceil\frac{n}{3}\right\rceil=k$ even labels used, so $\pr(P_n^2)$ is the $(n-k)^{\rm th}$ odd number, which is $2(n-k)-1=4k-1$.  The other two cases follow similarly to give us the value of $\pr(P_n^2)$.
\end{proof}

We next consider the square of the cycle, $C_n^2$, for $n\geq 4$.  Seoud and Youssef~\cite{SY} showed this graph to be not prime when $n\geq 4$ and that its independence number is $\displaystyle \left\lfloor\frac{n}{3}\right\rfloor$, which will be the maximum allowable even labels.
Note the only difference compared to the squared path graph is the additional edges $v_1v_{n-1}, v_1v_n$, and $v_2v_n$.  These do require some alterations to our labeling of~$P_n^2$ in some cases to maintain the coprime property, which results in the following for the minimum coprime number.
See Figure~\ref{c82} for an example of a minimum coprime labeling of $C_8^2$.

\begin{figure}[htb]
\begin{center}
	\includegraphics[scale=.9]{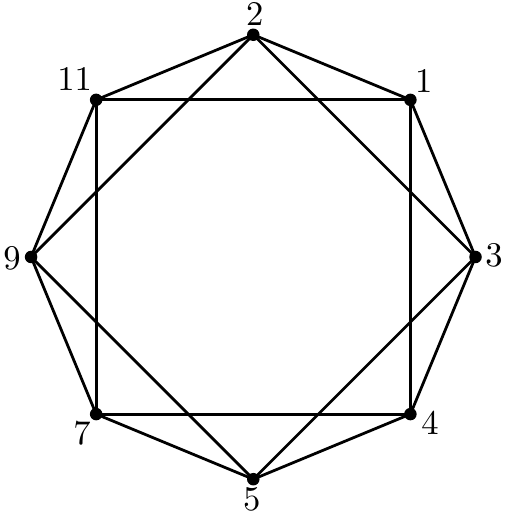}
\end{center}
\caption{Minimum coprime labeling of $C_8^2$}\label{c82}
\end{figure}

\begin{theorem}\label{pr_cycle_squared}
Let $n\geq 4$.  The minimum coprime number of $C_n^2$ is given by 
$$\pr(C_n^2)=\begin{cases}
4k-1 & \text{if } n=3k \\
4k+1 & \text{if } n=3k+1\\
4k+3 & \text{if } n=3k+2\\
\end{cases}$$
\end{theorem}

\begin{proof}
For the case of $n=3k$, the labeling of vertices $v_1,\ldots,v_n$ using the sequence $\{x_1,\ldots,x_n\}$ still provides a minimum coprime labeling as it did with the $P_n^2$.  This is because $x_1=2$ and $x_2=1$ while the final two labels are $x_{n-1}$ and $x_n$ being odd, so the additional three new edges maintain our coprime property.

The other two cases cannot simply use the same labeling as the squared path graph since the independence number is $1$ larger for $C_n^2$.  When $n=3k+1$, the edge $v_1v_n$ would result in a common factor of $2$ between the labels since $x_n$ is even.  Hence, we reassign $x_n=4k+1$, which is the lowest available odd label.  This label will now be relatively prime with its adjacent vertices whose labels are $1$, $2$, $4k-1$, and $4k-3$, making $\{x_1,\ldots,x_n\}$ a minimum coprime labeling.

Similarly, when $n=3k+2$, the edge $(v_1,v_{n-1})$ leads to adjacent vertices with even labels.  We correct this while keeping the labeling minimum by reassigning $x_{n-1}=4k+1$ and $x_n=4k+3$. Notice that $x_{n-1}$ and $x_n$ are the smallest available odd labels.  These are coprime with any adjacent vertices' labels since their neighbors are only labeled $1$, $2$, or an odd number within distance $4$ from themselves.  In each case, we obtain a minimum coprime labeling using the maximum amount of even labels and the smallest possible odd labels with the largest label being $4k-1$, $4k+1$, and $4k+3$ in the respective cases.
\end{proof}

Now we consider taking the third power of the path graph, $P_n^3$, which has additional edges from $v_i$ to $v_{i+3}$ for $i=1,\ldots, n-3$.  Since this graph contains $P_n^2$ as a subgraph, it is also not prime.  One can observe that the independence number for the cubed path is $\displaystyle \left\lceil\frac{n}{4}\right\rceil$, which will determine how many even labels can be placed on its vertices.  Note that when $n\leq 4$, the path cubed is simply a complete graph.

\begin{theorem}\label{pr_path_cubed}
Let $n\geq 5$. The minimum coprime number of $P_n^3$ is given by
$$\pr(P_n^3)=\begin{cases}
6k-1 & \text{if } n=4k\text{ or }4k+1 \\
6k+1 & \text{if } n=4k+2\\
6k+3 & \text{if } n=4k+3\\
\end{cases}$$
\end{theorem}

Similar to our construction for the coprime labeling of $P_n^2$, we define a sequence of finite length, with the subsequent terms determined by shifting the initial sequence.  For $P_n^3$, our initial sequence consists of 140 entries, and the shift is by 210 from $y_i$ to $y_{i+140}$.  Note that the sequence below is a repetition of even, odd, odd, odd entries with all possible odd numbers included and all even multiples of 3, 5, and 7 excluded.  Additionally, some other even numbers were removed to maintain the inequality in Lemma~\ref{distanceLemma2} regarding the distance between adjacent labels.   We define the labeling sequence as follows where vertex $v_n$ of $P_n^3$ is labeled by $y_n$
\begin{align*}
\{y_1,\ldots,y_{140}\}=\{&2,1,3,5,4,7,9,11,8,13,15,17,16,19,21,23,22,25,27,29,26,31,33,35,32,37,39,41,\\
&38,43,45,47,44,49,51,53,52,55,57,59,58,61,63,65,62,67,69,71,68,73,75,77,74,\\
&79,81,83,82,85,87,89,86,91,93,95,92,97,99,101,104,103,105,107,106,109,111,\\
&113,116,115,117,119,118,121,123,125,122,127,129,131,128,133,135,137,134,\\
&139,141,143,142,145,147,149,146,151,153,155,152,157,159,161,158,163,165,\\
&167,164,169,171,173,172,175,177,179,176,181,183,185,184,187,189,191,188,\\
&193,195,197,194,199,201,203,202,205,207,209\}
\end{align*}
and $y_i=y_{i-140}+210$ for $i> 140$.

We prove this theorem as was done for squared paths using a sequence of lemmas.
\begin{lemma}\label{distanceLemma2}
For adjacent vertices $v_i$ and $v_j$ in $P_n^3$ for some $1\leq i,j\leq n$, the labels in the sequence $\{y_1,\ldots,y_n\}$ satisfy $|y_i-y_j|\leq 9$.
\end{lemma}

\begin{proof}
The power of $P_n^3$ creates a neighborhood for each vertex $v_i$ of $\{v_{i-3},v_{i-2},v_{i-1},v_{i+1},v_{i+2},v_{i+3}\}$ with $i=4,\ldots, n-3$.  Through careful inspection of the $140$ initial labels, the maximum distance between labels of a vertex $v_i$ with $i\leq 137$ and its neighbor is $9$, which is attained for example by vertices $v_9$ and $v_{12}$ having labels 8 and 17, respectively.  

We next examine adjacent vertices with labels with one from the initial sequence and one from the shifted sequence.  Since $y_{138}=205$, $y_{139}=207$, $y_{140}=209$, $y_{141}=212$, $y_{142}=211$, and $y_{143}=213$, the adjacent labels with greatest distance apart are $|y_{138}-y_{141}|=7$, which is within the desired distance.

As in Lemma~\ref{distanceLemma1}, adjacent vertices with indices larger than $140$ maintain the same distance for their labels as their corresponding vertices with indices between $1$ and $143$, making the inequality hold for all adjacent vertices.
\end{proof}

\begin{lemma}\label{coprimeLabel2}
The labels $\{y_1,\ldots,y_n\}$ are a coprime labeling of $P_n^3$.
\end{lemma}
\begin{proof}
By Lemma~\ref{distanceLemma2}, the distance between adjacent labels is at most $9$; thus, no adjacent labels have a common factor of $10$ or higher.  The construction of the sequence results in even labels being four indices apart and hence not adjacent, so no labels share a common factor of $2$.  The even labels were chosen to not contain a factor of $3$, $5$, or $7$, and any pair of odd labels that both contain a multiple of $3$, $5$, or $7$ have indices that are at least $4$ apart, so they are not adjacent.  Thus, each pair of adjacent labels are coprime.
\end{proof}

\begin{lemma}\label{minimumCoprime2}
If $\max(y_1,\ldots, y_n)$ is odd, then the labels $\{y_1,\ldots,y_n\}$ are a minimum coprime labeling of $P_n^3$.
\end{lemma}
\begin{proof}
Since the independence number of $P_n^3$ is $\displaystyle \left\lceil\frac{n}{4}\right\rceil$, our sequence of labels uses the maximum number of even labels.  Therefore, the fact that we have used every odd number up to $\max(y_1,\ldots, y_n)$ implies that we have achieved a minimum coprime labeling.
\end{proof}

As with $P_n^2$, we define an altered labeling sequence for the case of $\max(y_1,\ldots, y_n)$ being even to switch out this maximum even label to ensure the maximum is odd.  We create our new sequence $\{y_i^*\}_{i=1}^n$ by defining $y_1^*=14$, $y_m^*=2$, and $y_i^*=y_i$ for all $i\in \{2,\ldots,n\}\setminus\{m\}$,
where $y_m$ was the maximum label of $\{y_1,\ldots,y_n\}$.  It can be observed from the initial sequence $\{y_1,\ldots,y_{140}\}$ that $m=n$ or $n-1$.

\begin{lemma}
If $\max(y_1,\ldots, y_n)$ is even, then the labels $\{y_1^*,\ldots,y_n^*\}$ are a minimum coprime labeling of $P_n^3$.
\end{lemma}

\begin{proof}
Through examination of the sequence $\{y_1,\ldots,y_{140}\}$, we see that the first even value that is larger than all preceding values is at $y_{69}=104$ when $n=69$ or $70$.  Replacing a maximum even label with $2$ and the first label with $14$ will result in there now being an odd label as the maximum.  
This allows us to apply Lemma~\ref{minimumCoprime2} if our labeling remains a coprime labeling.  Reassigning the label of $v_1$ to be $14$ maintains our coprime property since it is only adjacent to labels $1, 3$, and $5$.  The vertex $v_m$, whose maximum label was swapped for the label 2, is only adjacent to vertices with odd labels; hence, the adjacent labels remain relatively prime.
\end{proof}

\begin{proof}[Proof of Theorem~{\normalfont\ref{pr_path_cubed}}.]
The preceding lemmas have proven that $P_n^3$ has a mimimum coprime labeling using either $\{y_1,\ldots,y_n\}$ or $\{y_1^*,\ldots,y_n^*\}$, leaving us to verify that the correct minimum coprime number was attained.  The labels consist of $\displaystyle \left\lceil\frac{n}{4}\right\rceil$ even labels, so we consider the cases of $n \pmod{4}$.  For the case of $n=4k$, $\displaystyle \left\lceil\frac{n}{4}\right\rceil=k$ even labels were used.  Therefore, $\pr(P_n^3)$ is the $(n-k)^{\rm th}$ odd number, which is $2(n-k)-1=6k-1$.  The other three cases follow similarly to find their minimum coprime numbers.
\end{proof}

We next demonstrate a minimum coprime labeling of $C_n^3$, which we note has a independence number of $\displaystyle \left\lfloor\frac{n}{4}\right\rfloor$.  Also observe that $C_n^3=K_n$ for $n\leq 7$.
\begin{theorem}\label{pr_cycle_cubed}
Let $n\geq 8$.  The minimum coprime number of $C_n^3$ is given by 
$$\pr(C_n^3)=\begin{cases}
6k-1 & \text{if } n=4k \\
6k+1 & \text{if } n=4k+1\\
6k+5 & \text{if } n=4k+2\\
6k+7 & \text{if } n=4k+3\\
\end{cases}$$
\end{theorem}

\begin{proof}
Before considering each case, it is important to note that the vertex $v_n$ is adjacent to $v_3$ in $C_n^3$ since their distance in the cycle graph is $3$.  Observe from our labeling sequence that $y_n$ is a multiple of $3$ if and only if $n=4k+3$.

For $n=4k$, the labeling $\{y_1,\ldots,y_n\}$ that was used for $P_n^3$ is a minimum coprime labeling of $C_n^3$.  The additional edges in $C_n^3$ that were not in $P_n^3$ have endpoints with relatively prime labels because since $y_3=3$ and $y_n$ is not a multiple of 3 as previously above, $y_2=1$, and $y_1=2$ with the three vertices $y_{n-2},y_{n-1},y_n$, all having odd labels.

When $n=4k+1$, the independence number being $\displaystyle \left\lfloor\frac{n}{4}\right\rfloor$ implies that the sequence $\{y_1,\ldots,y_n\}$ cannot be used for the labeling since it would include $k+1$ even labels.  Instead, we reassign $y_n=6k+1$, which is the smallest unused odd label. If $y_1$ is not $2$ then relabel $y_1=2$ as well. Note that this label is $4$ larger than $y_{n-2}=y_{4k-1}$, which is a multiple of $3$, so it is relatively prime with $y_3=3$.  This final label is also clearly relatively prime with the labels $2$ and $1$ of the vertices $v_1$ and $v_2$, in addition to the odd labels at $y_{n-3}$, $y_{n-2}$, and $y_{n-1}$ since it is not a multiple of 3, resulting in a minimum coprime labeling.

Assuming $n=4k+2$, as in the previous case, there are too many even labels that requires a reassignment of the last even label to be $y_{n-1}=6k+1$.  If $y_1$ is not $2$ then relabel $y_1=2$ as well. This label is again relatively prime with $y_1$ and $y_2$.  The last label, $y_n$, which originally was also $6k+1$, cannot be reassigned to be the next smallest odd label of $6k+3$, because this is a multiple of $3$ with $v_n$ being adjacent to $v_3$.  Furthermore, labels that are multiples of $3$ cannot be shifted in any way to accommodate $6k+3$ because of the independence number being $\displaystyle \left\lfloor\frac{n}{4}\right\rfloor=k$ and the fact that there are already $k$ multiples of 3 in our labeling sequence.  Thus, we set $y_n=6k+5$, the smallest possible label that is not even or a multiple of 3.  Since it is coprime with the odd labels at $y_{n-1}$, $y_{n-2}$, and $y_{n-3}$, as well as $y_1$ and $y_2$, we have a minimum coprime labeling. 

We use the same reasoning for the $n=4k+3$ to reassign the labels $y_{n-2}=6k+1$, $y_{n-1}=6k+5$ (to avoid the multiple of $3$), and $y_n=6k+7$.  The labeling is minimum because the independence number limits the number of evens and multiples of 3 that can be used.
\end{proof}

The next logical step would be to generalize our constructions for $P_n^k$ and $C_n^k$ or at least continue with finding a minimum coprime labeling of $P_n^4$.  However, to keep even labels spaced out enough, a sequence for $P_n^4$ would require repetition of the pattern of even, odd, odd, odd, odd.  Using the smallest possible odd numbers fails quickly though as it would begin $2, 1, 3, 5, 7, 4, 9$, resulting in the labels $3$ and $9$ being adjacent due to their vertices being distance 4 apart.  Having to frequently skip odd numbers within our label sequence greatly increases the difficulty of finding minimum coprime labeling of $P_n^k$ for $k\geq 4$, so we leave this as an open problem in Section~\ref{remarks}.

\section{The Join of Paths and Cycles}\label{joins}

In this section we will establish the coprime labeling number for the join of two cycles, two paths, or a cycle and a path. As we did in Section~\ref{powers}, we will use the results on paths to find solutions for the join of cycles. In \cite{SDE}, it was shown that $P_n+K_1=P_n+P_1$ is prime, $P_n+\overline{K}_2$ is prime if and only if $n$ is odd, and $P_n+\overline{K}_m$ is not prime for $m\geq 3$.

We will exploit the size of the gap between prime numbers in order to ensure we can form a minimum coprime labeling. We define the gap between two primes as $g(p_n)=p_{n+1}-p_n$. By the prime number theorem, for all $\varepsilon>0$, there exists an integer $N$ such that for all $n>N$, $g(p_n)<\varepsilon p_n$. More specifically, particular values for $\varepsilon$ and $N$ are mentioned in \cite{R}.

\begin{theorem}\label{gap}
{\normalfont\cite{N,RW,S}} Let $\varepsilon>0$. For any positive integer $n>N$, $g(p_n)<\varepsilon p_n$ where $(N,\varepsilon)\in\{(9,1/5),(118,1/13),(2010760,1/16597)\}$. 
\end{theorem}

We first investigate a minimum coprime labeling of $P_m+P_2$. We know from \cite{SY} that $P_m+P_2$ is prime when $m$ is odd, so we focus on when $m$ is even. 

\begin{theorem}\label{pm+p2}
Let $m$ be a positive integer. If $m$ is even then $\pr(P_m+P_2)=m+3$ 
\end{theorem}

\begin{proof}
Since it was shown in \cite{SY} that $P_m+\overline{K}_2$ is not prime when $m$ is even, by Observation~\ref{primeFromPrime}, $P_m+P_2$ is not prime. 
Now we must find $\pr(P_m+P_2)$. 
When $m=2$, $P_m+P_2=K_4$, and so $\pr(P_2+P_2)=5$ by Proposition~\ref{complete}. It can be seen in Figure~\ref{pmp3} that $\pr(P_4+P_2)=7$, $\pr(P_6+P_2)=9$, and $\pr(P_8+P_2)=11$.

Suppose that $m\geq 10$ is even. We label the vertices of $V(P_m)=\{v_1,\ldots,v_m\}$ using the sequence
\[
2,m+3,4,3,8,5,6,9,10,p_4,12,p_5,14,15,16,p_6,18\ldots,m+1.
\]
Notice that $p_i$ is the label of $v_{p_{i+1}-1}$ for $i\in\{4,5,\ldots\}$.
By Theorem~\ref{gap}, $m$ is large enough so that $p_{i+1}<1.2 p_{i}$. Since the smallest number that is divisible by $p_{i}$ that is not $p_{i}$ is $2p_{i}$,  $p_{i}$ is relatively prime with the labels adjacent to $v_{p_{i+1}-1}$ for all $i\in\{4,5,\ldots\}$. Let $p'$ be the largest prime in the sequence $2,3,4,\ldots,m+1$. 
Then we label the vertices in $V(P_2)=\{u_0,u_1\}$ using the labels $1$ and $p'$. It is clear that this is a coprime labeling. Since we use all odd numbers less than the largest odd number label, it is also a minimum coprime labeling. 
\end{proof}

See Figure~\ref{pmp3} for an example of the labeling in the following theorem.

\begin{figure}[htb]
\begin{center}
	\includegraphics[scale=1.1]{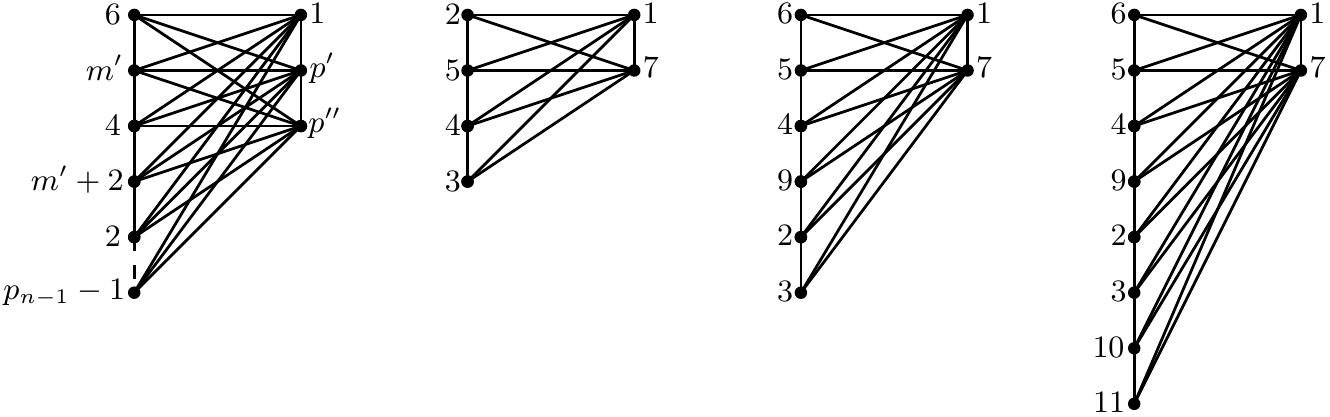}
\end{center}
\caption{Minimum coprime labeling of $P_m+P_3$ (left), $P_4+P_2$ (middle left), $P_6+P_2$ (middle right), and $P_8+P_2$ (right)}\label{pmp3}
\end{figure}

\begin{theorem}\label{pm+p3}
Let $m$ be a positive integer. Then $\pr(P_{m}+ P_3)= m+4+\frac{1-(-1)^m}{2}$.
\end{theorem}

\begin{proof}
Let $V(P_m)=\{v_1,\ldots,v_m\}$.
It is clear that $\pr(P_1+P_3)=4$. By Proposition~\ref{complete}, $\pr(P_2+P_3)=7$. By letting the sequence $2,3,4,9,8,11$ be the labels for $v_1,\ldots,v_6$ respectively and labeling the remaining vertices in $P_3$ with $1,5,7$, $\pr(P_3+P_3)=7$, $\pr(P_4+P_3)=\pr(P_5+P_3)=9$, and $\pr(P_6+P_3)=11$. 
By letting the sequence $6,5,4,3,2,9,8,13,10$ be the labels for $v_1,\ldots,v_9$ respectively and labeling the remaining vertices in $P_3$ with $1,7,11$, $\pr(P_7+P_3)=11$, $\pr(P_8+P_3)=\pr(P_9+P_3)=13$.
By letting the sequence $6,5,4,3,2,9,10,7,8,15,14$ be the labels for $v_1,\ldots,v_{11}$ respectively and labeling the remaining vertices in $P_3$ with $1,11,13$, $\pr(P_{10}+P_3)=\pr(P_{11}+P_3)=15$.

Suppose that $m\geq 12$. We label the vertices of $V(P_m)=\{v_1,\ldots,v_m\}$ using the sequence
\[
2,m',6,5,12,7,4,m'+2,8,9,10,3,14,15,16,p_5,18,p_6,20,21,22,p_7,\ldots,m+1
\]
where $m'$ is the smallest odd number larger than $m+1$. Depending on whether $m+1$ is even or odd will determine whether $m'$ is $1$ or $2$ larger than $m+1$. Notice that $v_{p_7-1}$ is labeled $p_5$, $v_{p_8-1}$ is labeled $p_6$, $v_{p_9-1}$ is labeled $p_7$, and so on. 
Note that if $3\mid m'$, swap $m'$ and $m'+2$ within the sequence.
By Theorem~\ref{gap}, $m$ is large enough so that $p_i<1.44 p_{i-2}$, thus $p_{i-2}$ is relatively prime with the labels on either side of $v_{p_i-1}$ for all $i\in\{7,8,\ldots\}$. Let $p'$ and $p''$ be the two largest primes in the sequence $2,3,4,\ldots,m+1$. 
Then we label the vertices in $V(P_3)=\{u_0,u_1,u_2\}$ using the sequence $1,p',p''$. It is clear that this sequence is a coprime labeling. 
Since we used all of the odd numbers less than our largest odd prime label when labeling $P_m$, it is also a minimum coprime labeling. 
\end{proof}

\begin{theorem}\label{pmp4}
Let $m$ be a positive integer. Then $\pr(P_m+ P_4)= m+6+\frac{1-(-1)^m}{2}$.
\end{theorem}

\begin{proof}
It is clear that $\pr(P_1+P_4)=5$. By Theorems~\ref{pm+p2} and \ref{pm+p3}, $\pr(P_2+P_4)=7$ and $\pr(P_3+P_4)=9$. 
By letting the sequence $2,3,4,9,8,13,6$ be the labels for $v_1,\ldots,v_7$ respectively and labeling the remaining vertices with $1,5,7,11$, $\pr(P_m+P_4)=m+6+\frac{1-(-1)^m}{2}$ for $m\in\{4,5,6,7\}$. 
By letting the sequence $2,5,4,3,10,9,8,5,6,17,16,19,18$ be the labels for $v_1,\ldots,v_{13}$ respectively and labeling the remaining vertices with $1,7,11,13$, $\pr(P_m+P_4)=m+6+\frac{1-(-1)^m}{2}$ for $m\in\{8,9,\ldots,13\}$.

Suppose that $m\geq 14$. We label the vertices of $V(P_m)=\{v_1,\ldots,v_m\}$ using the sequence
\[
6,m',8,m'+2,2,3,4,m'+4,10,9,14,5,12,15,16,p_4,18,p_5,20,21,\ldots
\]
where $m'$ is the smallest odd number larger than $m+1$. 
Whether $m+1$ is even or odd will determine if $m'$ is $1$ or $2$ larger than $m+1$. 
Notice that $v_{p_8-1}$ is labeled $p_5$, $v_{p_9-1}$ is labeled $p_6$, and so on. This pattern continues until $m+1$. 
Note that at most two of $m'$, $m'+2$, and $m'+4$ can be divisible by $3$ or~$5$, or at most one being divisible by both $3$ and $5$.  In each case, these three labels can be rearranged to maintain the relatively prime condition with their adjacent labels.

By Theorem~\ref{gap}, $m$ is large enough so that $p_i<1.728 p_{i-3}$ and so $p_{i-3}$ is relatively prime with the numbers on either side of the position of $p_i$ for all $i\in\{7,8,\ldots\}$. Let $q_1$, $q_2$, $q_3$ be the three largest primes in the sequence $2,3,4,\ldots,m+1$. 
Then we label the vertices in $P_4=u_0,u_1,u_2,u_3$ using the sequence $1,q_1,q_2,q_3$. It is clear that this sequence is a coprime labeling. 
Since we used all of the odd numbers less than our largest odd prime label when labeling $P_m$, it is also a minimum coprime labeling. 
\end{proof}

It is likely the case that $\pr(P_m+P_n)$ satisfies the equality below for $n,m\geq 5$, but the task of completing these cases is better left to a computer.

\begin{theorem}\label{pmpn<10}
Let $m>118$ and $n\leq 10$ be positive integers. Then $\pr(P_m+ P_n)= m+2n-2+\frac{1-(-1)^m}{2}$.
\end{theorem}

\begin{proof}
By Theorems~\ref{pm+p2}, \ref{pm+p3}, and \ref{pmp4}, let $5\leq n\leq 10$. We label the vertices of $P_m=v_0,v_1,\ldots,v_m$ using the sequence
\begin{align*}
2,&x_1,4,x_2,8,x_3,16,x_4,32,x_5,64,x_6,10,x_7,20,x_8,40,x_9,50,7,6,5,12,31,18,37,24,17,30,19,36,11,\\
42,&25,48,35,54,49,60,13,66,65,72,23,78,29,70,3,14,9,28,15,56,39,22,21,26,27,34,33,38,45,44,\\
51,&46,55,52,57,58,63,62,75,68,77,74,69,76,p_{13},80,81,82,p_{12},84,85,86,87,88,p_{14},89,90,91,92,93,\\
94,&95,96,p_{15},98,99,100,p_{16},102,p_{17},104,\ldots
\end{align*}
where $x_1,x_2,\ldots,x_9$ are the $n-1$ smallest odd integers larger than $m+1$. 
Depending on whether $m+1$ is even or odd will determine whether $\min(\{v_1,\ldots,v_n\})$ is $1$ or $2$ larger than $m+1$. 
Notice that at most two of the integers in the sequence $x_1,\ldots,x_9$ are divisible by $5$. So there are $4$ integers not divisible by $5$ that we will use to label the vertices $v_{12},v_{14},v_{16},v_{18}$.

If $n<10$ then $10-n$ of the labels in $\{x_1,x_2,\ldots,x_9\}$ will use $10-n$ of the smallest primes in $\{p_{14},p_{15},\ldots,p_{22}\}$ and the placement of the primes will shift by $10-n$ prime positions lower than they currently are in the sequence above. By Theorem~\ref{gap}, $m$ is large enough so that $p_i< (14/13)^9 < p_{i-(n-10)}$ and so $p_{i-(n-10)}$ is relatively prime with the labels on either side of $v_{p_i-1}$ for all $i\in\{14,15,\ldots,\}$. Let $q_1,q_2,\ldots,q_{n-1}$ be the $n-1$ largest primes in the sequence $2,3,4,\ldots,m+1$. 
Then we label the vertices in $V(P_n)=\{u_0,u_1,\ldots,u_{n-1}\}$ using the sequence $1,q_1,q_2,\ldots,q_{n-1}$. It is clear that this sequence is a coprime  labeling. 
Since we used all of the odd numbers less than our largest odd prime label when labeling $P_m$, it is also a minimum coprime labeling. 
\end{proof}

By applying the methods shown in the proofs of Theorem~\ref{pmp4} and \ref{pmpn<10} and with more precise values for $N$ and $\varepsilon$, the authors believe the following conjecture to be true. The difficulty in using the methods above would be coming up with a sequence of at most $N$ integers that would be coprime and use the integers in the sequence $2,3,\ldots,N$ save for the primes which would be shifted appropriately. 

\begin{conjecture}
Let $m>\varepsilon$ and $n\leq N$. Then $\pr(P_m+ P_n)\leq m+2n-2+\frac{1-(-1)^m}{2}$.
\end{conjecture}

Consequently, a stronger conjecture may be posed based solely on the size of $m+n$.

\begin{conjecture}
Let $m$ and $n$ be positive integers such that $m\geq n$. Then $\pr(P_m+P_n)=m+2n-2+\frac{1-(-1)^m}{2}$.
\end{conjecture}

We now move our discussion to the join of two cycles. We will use our work on the join of two paths to find the following results with relative ease. 
Notice that by Observation~\ref{primeFromPrime}, $C_m+C_n$, $C_m+P_n$ are not prime labelings when $m\geq n\geq 2$ or $n\geq m\geq 2$. 

\begin{theorem}\label{cm+cn}
Let $m$ and $n$ be positive integers such that $m\geq n$ and $n\leq 10$ and $m>118$ when $n\geq 5$. Then $\pr(C_m+C_n)=m+2n-1-(-1)^m$. 
\end{theorem}

\begin{proof}
If $m$ is even then adding two edges to graphs with labelings formed in Theorems~\ref{pm+p2}, \ref{pm+p3}, \ref{pmp4}, and \ref{pmpn<10}, the label on $v_1$ (it is $2$ in each case) is relatively prime with any odd integer and $1$ is relatively prime with everything, so the result follows. 

Suppose that $m$ is odd. 
The labels on $C_n$ still satisfy the coprime property. If $m$ is not prime then we can reassign the labeling on $v_m$ to the next available odd number which will only increase the coprime labeling number by $2$ from $\pr(P_m+P_n)$, but this is necessary since the independence number of $C_m$ is $\frac{m-1}{2}$. If $m$ is prime, this reassignment will still form a coprime labeling since $(14/13)^9$ is smaller than $2$ by enough of a margin that $m+2n$ will be relatively prime with $m$. 
Since we used all of the odd numbers less than our largest odd prime label when labeling of $P_m$, it is also a minimum coprime labeling. 
\end{proof}

The following corollary is directly from combining Observation~\ref{primeFromPrime} with Theorem~\ref{cm+cn}.

\begin{corollary}
Let $m$ and $n$ be positive integers such that either $m\leq 10$ or $n\leq 10$. If $m\geq n$ then $\pr(C_m+P_n)=m+2n-1-(-1)^m$. If $n>m$ then $\pr(C_m+P_n)=n+2m-2+\frac{1-(-1)^m}{2}$.
\end{corollary}

\section{Concluding Remarks}\label{remarks}

We conclude by posing several open questions regarding minimum coprime numbers.

\begin{question}
Can the minimum coprime number be determined for $P_n^k$ and $C_n^k$ for $k\geq 4$?
\end{question}

\begin{question}
Trees and grid graphs are conjectured to be prime, meaning their minimum coprime number would match their order.  Can the bound shown by Salmasian in {\normalfont\cite{S2}} that $\pr(T)\leq 4n$ for a tree $T$ of order $n$ be improved, and can a similar upper bound be found for the grid graph $P_m\times P_n$?
\end{question}

\begin{question}
Berliner et al.~{\normalfont\cite{Berliner}} investigated the minimum coprime number of $K_{n,n}$, but were not able to determine this number for all $n$.  Does there exist a formula for $\pr(K_{n,n})$ and can one determine the minimum coprime number more generally for all complete bipartite graphs $K_{m,n}$ in the cases in which there is no prime labeling?
\end{question}

\begin{question}
Many graphs that are not always prime are left to study in terms of minimum coprime labelings, such as M{\"o}bius ladders and $K_{1,n}+K_2$.  Can their minimum coprime numbers be determined?
\end{question}

\begin{question}
Is the following true: If $mn\leq p_{n-1}-n-1$ then $\pr(K_n\odot \overline{K}_m)=p_{n-1}$?
\end{question}

\begin{question}
Is the following true: If $mn>p_{n-1}-n-1$ then $\pr(K_n\odot \overline{K}_m)=mn+n+1$?
\end{question}

\bibliographystyle{amsplain}
\bibliography{vdec}

\end{document}